\documentclass[12pt, reqno]{amsart}
\usepackage{amsmath}
\usepackage{amssymb}
\usepackage{amsfonts}
\usepackage{color}
\usepackage{graphicx}
\usepackage[margin=2.9cm]{geometry}

\newcommand{\R}{\ensuremath{\mathbb{R}}}
\newcommand{\C}{\ensuremath{\mathbb{C}}}

\newcommand{\N}{\ensuremath{\mathbb{N}}}

\newcommand{\F}{\ensuremath{\mathbb{F}}}

\newtheorem{theorem}{Theorem}
\newtheorem{proposition}[theorem]{Proposition}

\newtheorem{lemma}[theorem]{Lemma}

\title[Polynomial solutions of polynomial Riccati equations] {The number of
polynomial solutions \\ of polynomial Riccati equations}

\author[A. Gasull]{Armengol Gasull}
\address{Departament de Matem\`{a}tiques, Universitat Aut\`{o}noma
de Barcelona, 08193 Bellater\-ra, Barcelona, Catalonia, Spain}
\email{gasull@mat.uab.cat}

\author[J. Torregrosa]{Joan Torregrosa}
\address{Departament de Matem\`{a}tiques, Universitat Aut\`{o}noma
de Barcelona, 08193 Bellater\-ra, Barcelona, Catalonia, Spain}
\email{torre@mat.uab.cat}

\author[X. Zhang]{Xiang Zhang}
\address{Department of Mathematics and MOE-LSC, Shanghai Jiao Tong
University, Shanghai, 200240, P. R. China}
\email{xzhang@sjtu.edu.cn}



\begin{document}

\begin{abstract}
Consider real or complex polynomial Riccati differential equations $a(x) \dot
y=b_0(x)+b_1(x)y+b_2(x)y^2$ with all the involved functions being polynomials of degree at most
$\eta$. We prove that the maximum number of polynomial solutions is $\eta+1$ (resp. 2) when
$\eta\ge 1$ (resp. $\eta=0$) and that these bounds are sharp.

For real trigonometric polynomial Riccati differential equations with all the functions being
trigonometric polynomials of degree at most $\eta\ge 1$ we prove a similar result. In this case,
the maximum number of trigonometric polynomial solutions is $2\eta$ (resp. $3$) when $\eta\ge 2$
(resp. $\eta=1$) and, again, these bounds are sharp.

Although the proof of both results has the same starting point, the
classical result that asserts that the cross ratio of four different
solutions of a Riccati differential equation is constant, the
trigonometric case is much more involved. The main reason is that
the ring of trigonometric polynomials is not a unique factorization
domain.
\end{abstract}

\maketitle

\section{Introduction and statement of the main results}

Riccati differential equations
\begin{equation}\label{e1}
a(x)\dot y =b_0(x)+b_1(x)y+b_2(x)y^2,
\end{equation}
where the dot denotes the derivative with respect to the independent variable $x$, appear in all text
books of ordinary differential equations as first examples of nonlinear equations. It is renowned
that if one explicit solution is known then they can be totally solved, by transforming them into
linear differential equations, see for instance \cite{Re1972}. One of their more remarkable
properties is that given any four solutions defined on an open set $\mathcal{I}\subset\R$, where
$y_1(x)<y_2(x)<y_3(x)<y_4(x),$ there exists a constant $c$ such that
\begin{equation}\label{rr}
\frac{(y_4(x)-y_1(x))(y_3(x)-y_2(x))
}{(y_3(x)-y_1(x))(y_4(x)-y_2(x)) }=c \quad \text{for all}\quad
x\in\mathcal{I}.
\end{equation}
Geometrically this result says that the cross ratio of the four functions is constant.
Analytically, it implies that once we know three solutions any other solution can be given in a
closed form from these three. We will often use an equivalent version of this fact, see
Lemmas~\ref{l3} and \ref{l3-1}.

Historically, Riccati equation has played a very important role in the pioneer work of D.
Bernoulli about the smallpox vaccination, see \cite{B, D} and also appears in many mathematical
and applied problems, see \cite{Hi, PWW, Ra1936}. The main motivation of this paper came to us
reading the works of Campbell and Golomb (\cite{Ca1952, CG1954}), where the authors present
examples of polynomial Riccati differential equations with 4 and 5 polynomial solutions. The
respective degrees of the polynomials $a, b_0, b_1$ and $b_2$ in these examples are $3, 3, 2, 0$ and
$4, 4, 3, 0$, respectively. At that point we wonder about the maximum number of polynomial solutions
that Riccati differential equations can have. Quickly, we realize that an upper bound for this
maximum number does not exist for linear differential equations. For instance the equation $\dot
y=x$ has all its solutions polynomials $y=x^2/2+c.$ In fact it is very easy to prove that linear
equations have 0, 1 or all its solutions being polynomials.

Looking at some other papers we found that Rainville \cite{Ra1936} in 1936 proved the existence of
one or two polynomial solutions for a subclass of \eqref{e1}. Bhargava and Kaufman
\cite{BK1965, BK1966} obtained some sufficient conditions for equation \eqref{e1} to have
polynomial solutions. Campbell and Golomb \cite{Ca1952, CG1954} provided some criteria determining
the degree of polynomial solutions of equation \eqref{e1}. Bhargava and Kaufman \cite{BK1964}
considered a more general form of equations than \eqref{e1}, and got some criteria on the degree
of polynomial solutions of the equations. Gin\'e, Grau and Llibre \cite{GGL2011} proved that
polynomial differential equations
\begin{equation}\label{e2}
\dot y=b_0(x)+b_1(x)y+b_2(x)y^2+\ldots+b_n(x)y^n,
\end{equation}
with $b_i(x)\in \R[x]$, $i=0, 1, \ldots, n$, and $b_n(x)\not\equiv 0$ have
at most $n$ polynomial solutions and they also prove that this bound is sharp.

In short, to be best of our knowledge, the question of knowing the maximum number of polynomial
solutions of Riccati polynomial differential equations when $a(x)$ is nonconstant is open. We
believe that it is also interesting because it is reminiscent of a similar question of Poincar\'{e}
about the degree and number of algebraic solutions of autonomous planar polynomial differential
equations in terms of their degrees, when these systems have finitely many algebraic solutions.
Recall that in this situation, similarly that for the linear case, there are planar polynomial
equations having rational first integrals for which all solutions are algebraic. As far as we
know, Poincar\'{e}'s question is open even for planar quadratic differential equations.

The first result of this paper solves completely our question for real or complex polynomial
Riccati differential equations. To be more precise, we say that equation \eqref{e1} is a
\emph{polynomial Riccati differential equation of degree $\eta$} when $a, \, b_0, \, b_1, \, b_2 \in
\F[x],$ the ring of polynomials in $x$ with coefficients in $\F=\R$ or $\C$, and
$\eta:=\max\{\alpha, \beta_0, \beta_1, \beta_2\},$ where $\alpha=\deg a$, and $\beta_i=\deg b_i$ for
$i=0, 1, 2.$
\begin{theorem}\label{t1}
Real or complex polynomial Riccati differential equations \eqref{e1} of degree~$\eta$ and
$b_2(x)\not\equiv 0$, have at most $\eta+1$ $($resp. $2)$ polynomial solutions when $\eta\ge 1$
$($resp. $\eta=0)$. Moreover, there are equations of this type having any given number of
polynomial solutions smaller than or equal to these upper bounds.
\end{theorem}
Notice that this theorem shows that the examples of Campbell and Golomb quoted in the
introduction precisely correspond to Riccati equations of degrees $\eta\in\{3, 4\}$ with the
maximum number ($\eta+1$) of polynomial solutions.

As we will see in Proposition~\ref{examplej}, the simple Riccati equation $a(x) \dot y= \dot
a(x)y+y^2,$ with $a(x)$ a polynomial of degree $\eta$ and simple roots has exactly $\eta+1$
polynomial solutions.

\medskip

When the functions appearing in \eqref{e1} are real trigonometric polynomials of degree at most
$\eta$ we will say that we have a \emph{trigonometric polynomial Riccati differential equation of
degree $\eta$}. Recall that the degree of a real trigonometric polynomial is defined as the degree
of its corresponding Fourier series. We are interested in them, not as a simple generalization of
the polynomial ones, but because they appear together with Abel equations ($n=3$ in \eqref{e2})
in the study of the number of limit cycles on planar polynomial differential equations with
homogeneous nonlinearities, see for instance \cite{GL, L}. In particular, periodic orbits
surrounding the origin of these planar systems correspond to $2\pi$-periodic solutions of the
corresponding Abel or Riccati equations.

It is an easy consequence of relation \eqref{rr} that when
$a(x)$ does not vanish, $T$-periodic Riccati equations of class
$\mathcal{C}^1$ have either continua of $T$-periodic solutions or at
most two $T$-periodic solutions. When $a(x)$ vanishes the situation is
totally different. These equations are singular at the zeros of $a$.
Equations with this property are called \emph{constrained
differential equation} and the zero set of $a$ is named
\emph{impasse set}, see \cite{SZ}. In particular, the Cauchy problem
has no uniqueness on the impasse set and the behavior of the
solutions of the Riccati equation is more complicated. As we will
see, there are real trigonometric polynomial Riccati differential
equations with an arbitrary large number of trigonometric
polynomial solutions, which are $2\pi$-periodic solutions of the
equation. As we will see this number is bounded above in terms of
the degrees of the corresponding trigonometric polynomials defining
the equation.

In general, we will write real trigonometric polynomial Riccati differential equations as
\begin{equation}\label{e1R}
A(\theta)Y'=B_0(\theta)+B_1(\theta)Y+B_2(\theta)Y^2,
\end{equation}
where the prime denotes the derivative with respect to $\theta$. To
be more precise $A,$ $B_0,$ $B_1,$ $B_2\in \R_t[\theta]:=\R[\cos
\theta, \sin \theta]$ the ring of trigonometric polynomials in $\cos
\theta, \sin \theta$ with coefficients in $\R.$ In this case we have
also $\eta:=\max\{\alpha, \beta_0, \beta_1, \beta_2\},$ where
$\alpha=\deg A$, and $\beta_i=\deg B_i$ for $i=0, 1, 2.$ Our second
main result is:

\begin{theorem}\label{t2}
Real trigonometric polynomial Riccati differential equations \eqref{e1R} of degree $\eta\ge 1$ and
$B_2(\theta)\not\equiv 0$, have at most $2\eta$ $($resp. $3)$ trigonometric polynomial solutions
when $\eta\ge 2$ $($resp. $\eta=1)$. Moreover, there are equations of this type having any given
number of trigonometric polynomial solutions smaller than or equal to these upper bounds.
\end{theorem}

For instance, consider the degree 3 trigonometric polynomial Riccati equation~\eqref{e1R} with
$A(\theta)\!=\!5 \sin\theta+8 \sin(2 \theta)+5 \sin(3 \theta),$ $B_0(\theta)\!=\!0,$
$B_1(\theta)\!=\!2+6 \cos\theta+18 \cos(2 \theta)+10 \cos(3 \theta),$ and $B_2(\theta)\!=\!-1.$ It
has exactly six trigonometric polynomial solutions $Y_1(\theta)= 0,$
\begin{align*}
 Y_2(\theta)&= 10+16 \cos\theta+10 \cos(2 \theta), \\[1mm]
Y_3(\theta)&=1-2 \cos\theta+3 \sin(2 \theta)+ \cos(2 \theta), \\[1mm]
Y_4(\theta)&=1-2 \cos\theta-3 \sin(2 \theta)+ \cos(2 \theta), \\[1mm]
Y_5(\theta)&=-3-8 \sin\theta-2 \cos\theta-5 \sin(2 \theta)+5 \cos(2 \theta), \\[1mm]
Y_6(\theta)&=-3+8 \sin\theta-2 \cos\theta+5\sin(2 \theta)+5 \cos(2 \theta).
\end{align*}
This example is constructed following the procedure described after
the proof of Theorem~\ref{t2}. From our general study of \eqref{e1R}
it can be understood why these solutions cross  the impasse set, see
Lemma~\ref{l4-1}. This phenomenon is illustrated in
Figure~\ref{fig1}.

\begin{figure}[h]

\includegraphics{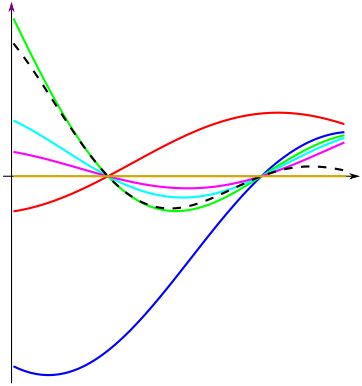}
\caption{The function $A(\theta)$ and the six  trigonometric polynomial solutions for a
trigonometric polynomial Riccati equation of degree 3. The graph of $A(\theta)$ is drawn as a
dashed line. } \label{fig1}
\end{figure}

As we will see, many steps of the proof of Theorem~\ref{t1} are
based on divisibility arguments in the ring of polynomials. A major
difference for proving Theorem~\ref{t2} is that the ring of
trigonometric polynomials is no more a \emph{Unique Factorization
Domain}. This can be seen for instance using the celebrated identity
$\cos^2\theta=1-\sin^2\theta = (1-\sin \theta)(1+\sin \theta).$ It
holds that $\cos \theta$ divides the right hand expression but it
does not divide either $1- \sin\theta$ nor $1+\sin\theta.$
Fortunately, divisibility reasonings in this ring can be addressed
by using the isomorphism
\begin{equation}\label{iso}
\begin{array}{lclc}
\Phi: & \R_t(\theta) & \longrightarrow & \R(x)\\
 & (\cos \theta, \, \sin \theta) & \longmapsto &
\left(\dfrac{1-x^2}{1+x^2}, \, \dfrac{2x}{1+x^2} \right),
\end{array}
\end{equation}
between the two fields $\R_t(\theta)=\R(\cos \theta, \sin \theta)$ and $\R(x).$ In fact, it can be
seen, see Lemma~\ref{l000}, that trigonometric polynomials of degree $\nu$ correspond in $\R(x)$
to rational functions of the form $f(x)/(1+x^2)^\nu,$ with $f(x)$ a polynomial of degree at most
$2\nu$ and coprime with $1+x^2.$ Then, for instance, the previous example $\cos^2\theta$ moves to
$(1-x^2)^2/(1+x^2)^2.$ This rational function decomposes as one of the two products
\[
\frac{1-x^2}{1+x^2}\times \frac{1-x^2}{1+x^2}, \quad
\frac{(1-x)^2}{1+x^2}\times\frac{(1+x)^2}{1+x^2},
\]
which precisely corresponds with the only two decompositions of $\cos^2\theta$ as a product of
trigonometric polynomials, which are $\cos\theta \times \cos\theta$ and $(1-\sin
\theta)\times(1+\sin \theta),$ respectively.

\medskip

The results for the polynomial Riccati equation are proved in Section~\ref{s2}. The proof of
Theorem~\ref{t2} about the trigonometric case will be given in Section~\ref{s3}. As we will see,
both proofs provide also more detailed information about the number and degrees of the polynomial
or trigonometric polynomial solutions.

\section{Polynomial Riccati equations}\label{s2}

This section is devoted to prove Theorem~\ref{t1}. Firstly we give some technical results and
secondly we organize the proof in three parts: the upper bound for $\eta\geq 2$
(Proposition~\ref{pr:case1}), the upper bound for $\eta\leq 1$ (Proposition~\ref{pr:case2}), and
the example that provides the concrete number of polynomial solutions up to reach the upper bound
(Proposition~\ref{examplej}). Recall that $a(x), b_0(x), b_1(x), b_2(x)$ in equation \eqref{e1} are
real or complex polynomials of degrees $\alpha, \beta_0, \beta_1, \beta_2,$ respectively.

Next result provides an upper bound on the degree of the polynomial
solutions of this equation.

\begin{lemma}\label{l1}
If $y_0(x)$ is a polynomial solution of equation \eqref{e1}, then $\deg y_0 \le \eta-\deg b_2.$
\end{lemma}

\begin{proof} By the assumption we have
\begin{equation}\label{e3}
a(x) \dot y_0(x)\equiv b_0(x)+b_1(x) y_0(x)+b_2(x) y_0^2(x), \quad
x\in \R.
\end{equation}
Set $\delta=\deg y_0$. We will prove that $\deg(b_2 y_0)\le \eta$.
If $\delta=0$ the result is trivial.
 Assume that $\delta>0,$ then the four components in \eqref{e3} have respectively the degrees
\begin{equation}
(\alpha+\delta-1;\, \, \beta_0, \, \, \beta_1+\delta, \,
\, \beta_2+2\delta). \label{eq:4degrees}
\end{equation}
Clearly we have
\[
\max\{\beta_0, \beta_1+\delta, \beta_2+2\delta\}\ge \alpha+\delta-1.
\]
Additionally, when the above maximum is taken only by one of the three numbers
then the equality holds.

\medskip

\noindent $\bullet$ [\emph{Case
$\max\{\beta_0, \beta_1+\delta, \beta_2+2\delta\}= \alpha+\delta-1$}]:
Then $\beta_2+\delta\le \alpha-1<\eta$. That is
$\deg(b_2\, y_0)=\beta_2+\delta < \eta$.

\medskip

\noindent$\bullet$ [\emph{Case
$\max\{\beta_0, \beta_1+\delta, \beta_2+2\delta\}> \alpha+\delta-1$}]:
We must have
\[
\beta_1+\delta=\beta_2+2\delta, \quad \text{ or } \quad
\beta_0=\beta_2+2\delta, \quad \text{ or } \quad
\beta_0=\beta_1+\delta.
\]
When $\beta_1+\delta=\beta_2+2\delta$, we have
$\beta_2+\delta=\beta_1\le \eta$. If $\beta_2+2\delta=\beta_0$ then
$\beta_2+\delta<\beta_0\le \eta$. Finally, when
$\beta_0=\beta_1+\delta$, we must have $\beta_2+2\delta\le
\beta_0=\beta_1+\delta.$ Then again $\beta_2+\delta< \eta$.
\end{proof}

Lemma~\ref{l1} has a useful consequence in proving Theorem~\ref{t1}
for a reduced case.

\begin{lemma}\label{l2}
If a polynomial equation~\eqref{e1} of degree $\eta$ has a
polynomial solution, then it is equivalent to an equation with
$b_0(x)\equiv 0$ and the same degree:
\begin{equation}\label{e1*}
a(x)\dot y=b_1(x)y+b_2(x)y^2.
\end{equation}
Moreover, each polynomial solution of equation~\eqref{e1} corresponds to a polynomial solution of
the new equation~\eqref{e1*}.
\end{lemma}

\begin{proof}
If $y_0$ is a solution of \eqref{e1}, the proof follows taking the
change of variable $w(x)=y(x)-y_0(x)$ and applying Lemma~\ref{l1}.
More concretely, equation \eqref{e1} becomes
\[
a(x)\dot w=(b_1(x)+2\, b_2(x)\, y_0(x))w+b_2(x)\, w^2,
\]
and $\deg(b_1+2b_2y_0)\le \max\{\deg b_1, \deg(b_2y_0)\}\le \eta$.
\end{proof}

The next result gives a criterion to relate the zeros of $a$ with the zeros of polynomial
solutions of equation \eqref{e1*}.

\begin{lemma}\label{l4}
If $x^*$ is a zero of a nonzero polynomial solution of equation
\eqref{e1*}, then $a(x)$ also vanishes at $x=x^*.$
\end{lemma}

\begin{proof} Let $y_1$ be a polynomial solution of \eqref{e1*}. Then we have
\[
a(x)\dot y_1(x)=(b_1(x)+b_2(x)y_1(x))y_1(x).
\]
The statement follows by equating the power of $(x-x^*)$ in the
above equation.
\end{proof}

We remark that in the later result even when $x^*$ is a zero of a
solution with multiplicity $\nu\ge2$, it may be a simple zero of
$a(x)$. For example, the equation $x\dot y=(\nu-x^\nu)y+y^2$ has the
solution $y=x^\nu$ with $x=0$ a zero of multiplicity $\nu$, but
$a(x)=x$ has $x=0$ as a simple zero. This example also illustrates
that the upper bound of the degree of polynomial solutions in
Lemma~\ref{l1} can be achieved.

Next lemma is the version of \eqref{rr} that we will use. From now
on, in this work, given two polynomials $f$ and $g$ we will denote
by $\gcd(f,g)$ their monic greater common divisor.

\begin{lemma}\label{l3}
Let $y_0= 0, \, y_1(x), \, y_2(x)$ be polynomial solutions of
\eqref{e1*} such that $y_1(x)\not\equiv 0$ and $y_2(x)\not\equiv 0.$
Set $y_1(x)=g (x)\tilde y_1(x)$ and $y_2(x)=g(x)\tilde y_2(x)$,
where $g=\gcd(y_1, y_2),$ and $\gcd(\tilde y_1, \, \tilde y_2)=1.$
Except the solution $y=0$, all the other solutions of equation
\eqref{e1*} can be expressed as
\begin{equation}\label{e123*}
y(x;c)=\frac{\tilde y_1(x) \, \tilde y_2(x)g(x)}{c\tilde
y_1(x)+(1-c)\tilde y_2(x)},
 \end{equation}
where $c$ is an arbitrary constant.
\end{lemma}
\begin{proof} Let $y$ be a nonzero solution of equation \eqref{e1*}.
The functions $z=1/y$, $z_1=1/y_1,$ and $z_2=1/y_2$ are solutions of
a linear differential equation and satisfy
\[
-a(x)\dot z=b_1(x)z+b_2(x), \quad -a(x)\dot z_1=b_1(x)z_1+b_2(x), \quad -a(x)\dot z_2=b_1(x)z_2+b_2(x).
\]
It follows that
\[
\frac{\dot z(x)-\dot z_1(x)}{z(x)-z_1(x)}=\frac{\dot z_2(x)-\dot z_1(x)}{z_2(x)-z_1(x)}.
\]
Consequently
\[
z(x)=z_1(x)+c(z_2(x)-z_1(x)),
\]
with $c$ an arbitrary constant. So, the general solution of \eqref{e1*} is
\[
y(x;c)=\frac{y_1(x) \, y_2(x) }{c y_1(x)+(1-c) y_2(x)},
\]
with $c$ an arbitrary constant. The proof ends substituting $y_1(x)=g(x)\, \tilde y_1(x)$ and
$y_2(x)=g\tilde y_2(x)$ into this last expression.
\end{proof}

\begin{lemma}\label{l5}
Denote by $N(g)$ the number of different zeros of $g$. The following
statements hold.
\begin{itemize}
\item[$(i)$] If for any value of $c$, $c\tilde y_1(x)+(1-c)\tilde y_2(x)$
is not a constant, then equation \eqref{e1*} has at most $N(g)+3$
polynomial solutions.
\item[$(ii)$] If there exists a value $c_0$ such that $c_0\tilde
y_1(x)+(1-c_0)\tilde y_2(x)$ is a constant, then
equation~\eqref{e1*} has at most $N(g)+4$ polynomial solutions.
\end{itemize}
\end{lemma}

\begin{proof} From Lemma~\ref{l3} the general solution of equation \eqref{e1*}
is \eqref{e123*}. Since $\gcd(\tilde y_1, \tilde y_2)=1$, in order that $y(x;c)$ is a polynomial
solution of equation \eqref{e1*} we must have $c=0$, or $c=1$, or $c\ne 0, \, 1 $ and $c\tilde
y_1+(1-c)\tilde y_2$ divides $g$. Since $\deg g\le \eta$, it follows that $g$ has at most $\eta$
zeros. We distinguish two cases depending on whether the denominator of \eqref{e123*} can be a
constant or not.

\medskip

\noindent $\bullet$ [\emph{Case of denominator never being a constant}]: For any $c$, the
polynomial $c\tilde y_1(x)+(1-c)\tilde y_2(x)$ is never a constant. Let $x^*$ be a root of $g$. In
order that $c\tilde y_1+(1-c)\tilde y_2$ divides $g$ for some $c$, we must have
\[
c\tilde y_1(x^*)+(1-c)\tilde y_2(x^*)=0.
\]
If $\tilde y_1(x^*)=\tilde y_2(x^*)$ then from the above relation we get that both must be zero,
but this is impossible because $\gcd(\tilde y_1, \tilde y_2)=1$. If $\tilde y_1(x^*)-\tilde
y_2(x^*)\ne 0$, solving this last equation gives $c^*=\tilde y_2(x^*)/(\tilde y_2(x^*)-\tilde
y_1(x^*))$. If $c^*\tilde y_1+(1-c^*)\tilde y_2$ divides $g$, we have the polynomial solution
$y(x;c^*)$ given in \eqref{e123*}. Thus we have at most $N(g)$ different values of $c$ such that
$c\tilde y_1+(1-c)\tilde y_2$ divides $g$. As a consequence, \eqref{e1} has at most $N(g)$ number
of polynomial solutions together with the solutions $0$, $y_1$ and $y_2$.

\medskip

\noindent $\bullet$ [\emph{Case of some constant denominator}]: There exists a $c_0$ such that
$c_0\tilde y_1(x)+(1-c_0)\tilde y_2(x)$ is constant. First we claim that this $c_0$ is unique. If
not, there are two numbers $c_1, c_2$ such that
\begin{align*}
c_1\tilde y_1(x)+(1-c_1)\tilde y_2(x)\equiv d_1, \\
c_2\tilde y_1(x)+(1-c_2)\tilde y_2(x)\equiv d_2,
\end{align*}
for some constants $d_1$ and $d_2$. Then we have
\begin{equation}\label{ecoef24}
(c_1d_2-c_2d_1)\tilde y_1(x)+(d_2-d_1-c_1d_2+c_2d_1)\tilde
y_2(x)\equiv 0.
\end{equation}
If the coefficients of $\tilde y_1$ and $\tilde y_2$ in \eqref{ecoef24} are not all zero, then one
of $\tilde y_1$ and $\tilde y_2$ is a multiple of the other. This is in contradiction with the
fact that $\gcd(\tilde y_1, \tilde y_2)=1$. If the coefficients in \eqref{ecoef24} are both zero,
we have $d_1=d_2$ and $c_1=c_2$. This proves the claim.

As in the proof of the previous case, we now have $N(g)+1$ possible values of $c$ for which
$g/(c\tilde y_1+(1-c)\tilde y_2)$ are polynomials. This implies that equation \eqref{e1*} has at
most $N(g)+4$ nonzero polynomial solutions.
\end{proof}

\begin{proposition}\label{pr:case1}
Let $a(x), b_i(x)$ be polynomials of degrees $\alpha=\deg a$,
$\beta_i=\deg b_i$, $i=1, 2$, and
$\eta=\max\{\alpha, \beta_1, \beta_2\}\geq 2$ with $b_2(x)\not\equiv
0$. Then equation \eqref{e1*} has at most $\eta+1$ polynomial
solutions.
\end{proposition}

\begin{proof}
We assume that equation \eqref{e1*} has three different polynomial
solutions $0,y_1, y_2$. Otherwise the statement follows. Recall that
$g=\gcd(y_1, y_2)$ and it is monic, see Lemma~\ref{l3}. Then, $\deg
g\le \max\{\deg y_1, \deg y_2\}\le \eta$ and $N(g)\le \eta$. We
will split our study in the three cases: $N(g)=\eta,$
$N(g)=\eta-1,$ or $N(g)\le \eta-2.$
\medskip

\noindent $\bullet$ [\emph{Case $N(g)=\eta$}]: Then $\deg g=\deg y_1=\deg y_2=\eta$. Hence
\[
y_1(x)=p\,g(x), \qquad y_2(x)=q\, g(x), \qquad p, q \in \F^*=\F\setminus\{0\}.
\]
It follows from Lemma~\ref{l3} that all solutions of equation
\eqref{e1*} are of the form
\[
y(x;c)=k_c \,g(x), \qquad k_c\in \F.
\]
This forces that $b_2(x)\equiv 0$, a contradiction with the assumption $b_2(x)\not\equiv 0$.

\medskip

\noindent $\bullet$ [\emph{Case $N(g)=\eta-1$}]:
 We split the proof in two cases:

\smallskip

\noindent $(i)$ \emph{$\deg g=\eta$ or $\deg y_1=\deg y_2=\eta-1$}:
We have $y_1=p\, g, \, y_2=q\, g$ with $p, q \in \F^*$. Similarly
with the case $N(g)=\eta$, we get $b_2(x)\equiv 0$, again a
contradiction.

\smallskip

\noindent $(ii)$ \emph{$\deg g=\eta-1$ and $\, \max(\deg y_1, \deg
y_2)=\eta$}: We have that $y_i= g \, \tilde y_i$ with $\deg \tilde
y_i\le1,$ $i=1, 2.$ Moreover, by Lemma~\ref{l4} it holds that $a= g
\, \tilde a$ with $\deg \tilde a\le 1.$

Since $y_1$ and $y_2$ are solutions of equation \eqref{e1*}, we get from
\[
a(x)\dot y_1=b_1(x)y_1+b_2(x) y_1^2, \quad a(x)\dot
y_2=b_1(x)y_2+b_2(x) y_2^2,
\]
that
\begin{equation}\label{e5}
b_2(x)=\frac{a(x)\big(y_1(x)\dot y_2(x)-y_2(x) \dot
y_1(x)\big)}{y_1(x)y_2(x)(y_2(x)-y_1(x))}= \frac{\tilde a(x)\big(\tilde y_1(x)\dot{\tilde
y}_2(x)-\tilde y_2(x) \dot{\tilde y}_1(x)\big)}{\tilde y_1(x)\tilde y_2(x)(\tilde y_2(x)-\tilde
y_1(x))}.
\end{equation}
By Lemma~\ref{l1}, $b_2(x)\equiv p$ for some constant number. Thus
\[
p\, \tilde y_1(x)\tilde y_2(x)(\tilde y_2(x)-\tilde y_1(x))=\tilde a(x)\big(\tilde
y_1(x)\dot{\tilde y}_2(x)-\tilde y_2(x) \dot{\tilde y}_1(x)\big).
\]
Notice that $\tilde y_1\dot{\tilde y}_2-\tilde y_2 \dot{\tilde y}_1$ is always a nonzero
constant, because $\, \max(\deg \tilde y_1, \deg \tilde y_2)=1$ and $\gcd(\tilde y_1, \tilde
y_2)=1$. Hence, in the above formula, the left-hand side has degree at least two while the
right-hand side has degree at most one, a contradiction.

Hence under all the above hypotheses the maximum number of polynomial solutions is $2\le \eta+1$,
for $\eta\ge 1$, as we wanted to prove.

\medskip

\noindent $\bullet$ [\emph{Case $N(g)\le\eta-2$}]: If $c\, \tilde
y_1(x)+(1-c)\, \tilde y_2(x)$ is nonconstant, for any value of $c$,
the result holds from statement $(i)$ of Lemma~\ref{l5}. If
$N(g)<\eta-2$ the result also follows from Lemma~\ref{l5}.

Hence we assume that $N(g)=\eta-2$ and that there exists a constant
$c_0,$ with $c_0\ne 0, \, 1$, such that
\[
c_0\tilde y_1(x)+(1-c_0)\tilde y_2(x)\equiv q\in\F^*.
\]
Note that if $q=0$, then $\tilde y_1$ and $\tilde y_2$ will have a
common factor, it is in contradiction with $\gcd(\tilde y_1, \tilde
y_2)=1$. We recall that $y(x;0)=y_1(x)$ and $y(x;1)=y_2(x)$. Now
equation~\eqref{e1*} also has the polynomial solution
\[
y_3(x)=\tilde y_1(x)\tilde y_2(x) g(x)/q.
\]
From $N(g)=\eta-2$ we get that $\deg g\ge \eta-2$.

Since any polynomial solution of equation \eqref{e1} has degree at most $\eta$, by the assumptions
and the expressions of $y_1$ and $y_2$ we get that $\tilde y_1$ and $\tilde y_2$ are polynomials
of the same degree $d\in\{0, 1, 2\}$. But it follows from the expression of $y_3$ that $d=2$ is not
possible, otherwise the polynomial solution $y_3(x)$ would have degree at least $\eta+2$, a
contradiction.

If $d=0$, then $y_1$ and $y_2$ are constant multiples of $g$.
Similarly to the proof of case $N(g)=\eta$ we get that $b_2(x)\equiv
0$, again a contradiction.

When $d=1$, if $\tilde y_1$ or $\tilde y_2$ has a zero which is not a zero of $g$, we assume
without loss of generality that $\tilde y_1$ satisfies this condition. Set
\[
h(x):=\gcd(y_1(x), y_3(x))=\tilde y_1(x)\, g(x).
\]
Then $N(h)=\eta-1$. In this situation we can start again our proof
taking $y_1$ and $y_3$ instead of $y_1$ and $y_2$ and we are done.

Finally, assume that $d=1$ and each zero of $\tilde y_1$ and $\tilde y_2$ is a zero of $g$. Then,
since any polynomial solution of equation \eqref{e1*} has degree at most $\eta$, we get from the
expression of $y_3$ that $\deg g\le \eta-2$. Thus $\deg g=\eta-2$. So we have
\begin{equation}\label{e8}
g(x)=(x-x_1)\cdots (x-x_{\eta-2}) \quad \text{ with } x_i\ne x_j \text{ for } 1\le i\ne j\le
\eta-2.
\end{equation}
Since the zeros of $\tilde y_1$ and $\tilde y_2$ are zeros of $g(x)$, we can assume that
\begin{equation}\label{e9}
\tilde y_1=p_1(x-x_1), \quad \tilde y_2=p_2(x-x_2) \quad \text{ with
} p_1, p_2 \, \, \text{ nonzero constants}.
\end{equation}
Thus in order that $y(x;c)$ with $c\ne 0, \, 1$, is a polynomial
solution of equation \eqref{e1*}, we get from \eqref{e123*},
\eqref{e8}, and \eqref{e9} that
\[
c\tilde y_1(x)+(1-c)\tilde y_2(x)=r_i(x-x_i) \quad \text{ with } i\in\{3, \ldots, \eta-2\},
\]
where each $r_i$ is a nonzero constant. Clearly the number of possible values of $c$ satisfying
these conditions is exactly $\eta-4$. This proves that equation \eqref{e1*} has $\eta$ polynomial
solutions, including $0$, $y_1, \, y_2$ and $y_3$.
\end{proof}

\begin{proposition}\label{pr:case2}
Let $a(x), b_i(x)$ be polynomials of degrees $\alpha=\deg a$,
$\beta_i=\deg b_i$, $i=1, 2$, and
$\eta=\max\{\alpha, \beta_1, \beta_2\}\leq 1$ with $b_2(x)\not\equiv
0$. Then equation \eqref{e1*} has at most $2$ polynomial solutions.
\end{proposition}

\begin{proof}
We prove the statement in two different cases: $\eta=0$ and
$\eta=1.$

\medskip

\noindent $\bullet$ [\emph{Case $\eta=0$}]: Equation \eqref{e1*} has constant coefficients, and it
can be written as
\[
\dot y=p \, y (y-q),
\]
with $p\ne 0$ and $q$ constants. This equation has two constant solutions for $q\ne 0$ and one
constant solution for $q=0.$ The proof of this case finishes, using Lemma~\ref{l1}, because any
polynomial solution should be constant.

\medskip

\noindent $\bullet$ [\emph{Case $\eta=1$}]: By contradiction we assume that equation \eqref{e1*}
has three polynomial solutions $0, \, y_1,$ and $y_2$. Lemma~\ref{l1} shows that $\max\{\deg y_1,
\deg y_2\}\le 1$. The proof is done with a case by case study on the degrees of $y_1$ and $y_2.$

\smallskip

\noindent $(i)$ If $y_1$ and $y_2$ are both linear, then they cannot
have a common zero, otherwise it follows from Lemma~\ref{l3} that
all solutions are constant multiple of $y_1$, and so $b_2(x)\equiv
0$, a contradiction. Since $y_1$ and $y_2$ have different zeros, it
follows from Lemma~\ref{l4} that $a$ has two different zeros. This
means that $a$ has degree at least $2$, again a contradiction with
$\eta=1$. These contradictions imply that equation \eqref{e1*} has
never two nonzero polynomial solutions of the given form.

\smallskip

\noindent $(ii)$ If $y_1$ and $y_2$ are both constants, we get from
Lemma~\ref{l3} that all solutions of equation~\eqref{e1} are
constants. And so $b_2(x)\equiv 0$, again a contradiction.

\smallskip

\noindent $(iii)$ Assume that one of $y_1$ and $y_2$ is linear and
another is a constant. Without loss of generality we assume that
$y_1$ is linear and $y_2(x)=p\ne 0$ a constant. Now we have by
Lemma~\ref{l4} that $a(x)=r\, y_1(x)$ with $r$ a nonzero constant.
From \eqref{e5} we have
\[
b_2(x)=\frac{r \dot y_1(x)}{y_1(x)-p},
\]
that is not a polynomial, because $\dot y_1$ is a nonzero constant. This contradiction shows that
equation \eqref{e1} has neither two nonzero polynomial solutions of the given form.
\end{proof}

\begin{proposition}\label{examplej} Let $x_i,$ $i=1, \ldots, \eta$ be different
values in $\F.$ For each $j\in\{2, \ldots, \eta+1\},$ consider the
polynomial $a(x)=-(x-x_1)^{\eta+2-j}(x-x_2)\cdots (x-x_{j-1})$ of
degree $\eta.$ Then the differential equation \eqref{e1} with
$b_0(x)\equiv0,$ $b_1(x)=\dot a(x),$ and $b_2(x)\equiv 1$
has exactly $j$ polynomial solutions.
\end{proposition}

\begin{proof}
Equation \eqref{e1} with $b_0(x)\equiv 0, \, b_1(x)=\dot a(x)$ and
$b_2(x)\equiv 1$ has the solution $y=0.$ When $y\ne 0$, \eqref{e1}
can be written as
\[
\frac{d}{dx} \left(\frac {a(x)} {y}\right)=-1.
\]
It has a general solution
\[
y(x)=\frac{-a(x)}{x-c}=\frac{(x-x_1)^{\eta+2-j}(x-x_2)\cdots (x-x_{j-1})}{x-c},
\]
with $c$ an arbitrary constant. For each $j\in\{2, \ldots, \eta+1\},$
choosing $c= x_i$, $i=1, \ldots, j-1$, we get exactly $j-1$ polynomial
solutions
\[
y_i=(x-x_1)^{\eta+2-j}\cdots (x-x_{i-1})\widehat{(x-x_i)} (x-x_{i+1}) \cdots (x-x_{j-1}), \quad i=1, \ldots, j-1,
\]
where $\widehat{(x-x_i)}$ denotes the absence of the factor. Then
equation \eqref{e1} has exactly $j$ polynomial solutions including
the trivial one $y=0$.
\end{proof}

\begin{proof}[Proof of Theorem~\ref{t1}]
When equation~\eqref{e1} has no polynomial solutions we are done. Otherwise we can apply
Lemma~\ref{l2} and restrict our analysis to equation~\eqref{e1*}. Propositions~\ref{pr:case1} and
\ref{pr:case2} provide the first part of the statement, which corresponds to the upper bound of
the statement when the degree is $\eta\geq 2$ and $\eta\leq 1,$ respectively. Proposition~\ref{examplej}
gives a concrete polynomial differential equation with exactly $j$ polynomial solutions for each
$j\in\{2, \ldots, \eta+1\}.$ The cases $j\in\{0, 1\}$ are trivial. These facts prove the second part
of the statement.
\end{proof}

\section{Trigonometric polynomial Riccati equations}\label{s3}

This section is devoted to prove Theorem~\ref{t2}. As in the previous section, first we give some
technical results. Since the proof of Theorem~\ref{t2} is different for $\eta\ge 2$ and $\eta=1$,
we distinguish these two cases.

Let $F(\theta)\in \R_t[\theta]$ be a real trigonometric polynomial,
recall that it is said that $F(\theta)$ has \emph{degree} $\nu$ if
$\nu$ is the degree of its associated Fourier series,
i.e.
\[
F(\theta)=\sum\limits_{k=-\nu}\limits^\nu f_ke^{k \theta \mathbf i},
\quad \text{ with } \mathbf i=\sqrt{-1},
\]
where $f_k=\overline f_{-k}\in\C$, $k\in \{1, \ldots, \nu\}$, and
$f_\nu$ is nonzero, see for instance \cite{Br2004, CGM2013}. As we
have already explained, $\R_t[\theta]$ is not a unique factorization
domain and this fact complicates our proofs. The next lemma provides
the image, by the isomorphism $\Phi$ given in \eqref{iso}, of the
ring $\R_t[\theta]$ in $\R(x)$, see for instance \cite[Lemma
10]{CGM2013} and \cite[Lemma 17]{GGL2013}. In fact, it is one of the
key points in the proof of Theorem~\ref{t2} because the map $\Phi$
moves a polynomial in $\sin \theta, \cos \theta$ into a rational
function in $x,$ such that the numerator has a unique decomposition
as a product of irreducible polynomials. Another minor difference
between this case and the polynomial one is that the degree remains
invariant by derivation when we consider trigonometric polynomials.

For sake of shortness, in this paper we do not treat the case of complex trigonometric
polynomial Riccati equations.

\begin{lemma}\label{l000}
Set $F(\theta)\in\R_t[\theta]$ with $\deg F=\nu$. Then
\[
\Phi(F(\theta))=\frac{f(x)}{(1+x^2)^\nu},
\]
with $\gcd(f(x), 1+x^2)=1$ and $\deg f\le 2\nu$. Conversely, any
rational function $g(x)/(1+x^2)^\nu$ with $g(x)$ an arbitrary
polynomial of degree no more than $2\nu$ can be written as a
trigonometric polynomial through the inverse change, $\Phi^{-1}.$
\end{lemma}

In equation \eqref{e1R}, set
\begin{equation}\label{eq:defab}
A(\theta)=\frac{a(x)}{(1+x^2)^\alpha}, \quad B_i(\theta)=\frac{b_i(x)}{(1+x^2)^{\beta_i}},
\end{equation}
where $\alpha=\deg A,$ $\deg a \le 2\alpha,$ $\beta_i=\deg B_i,$ and
$\deg b_i\le 2\beta_i$ for $i=0, 1, 2$. By the assumption of Theorem~\ref{t2} we have
$\eta=\max\{\alpha, \beta_0, \beta_1, \beta_2\}$.

Using Lemma~\ref{l000} we can transform the trigonometric polynomial Riccati differential equation
\eqref{e1R} into a polynomial Riccati differential equation \eqref{e1}, see the proof of
Lemma~\ref{l4-1}. Then we can apply Theorem~\ref{t1} to this polynomial differential equation. It
can be seen that in this way we can prove (we omit the details) that an upper bound of the
trigonometric polynomial solutions of equation \eqref{e1R} is $6\eta+1$. This upper bound is much
higher than the actual one given in Theorem~\ref{t2}.

Although the outline for proving Theorem~\ref{t2} will be the same
that we have followed in Section~\ref{s2}, we will see that in this
case the proof is much more involved.

We start providing an upper bound on the degree of trigonometric polynomial solutions
of the trigonometric polynomial equation \eqref{e1R}. Notice that, for this first result, we do not use the map $\Phi.$

\begin{lemma}\label{l1-1}
If $Y_0(\theta)$ is a real trigonometric polynomial solution of the real
trigonometric polynomial equation \eqref{e1R} of degree $\eta$, then
$\deg(Y_0(\theta))\le \eta-\deg(B_2(\theta)).$
\end{lemma}

\begin{proof} Given two trigonometric polynomials $P$ and $Q$ it holds that
 $\deg(P(\theta)Q(\theta))=\deg(P(\theta))+\deg(Q(\theta))$ and $\deg(P'(\theta))=\deg(P(\theta)).$
Set $\delta=\deg Y_0$ then the degrees of the four components in
\eqref{e1R} are respectively
\[
(\alpha+\delta;\, \, \beta_0, \, \, \beta_1+\delta, \, \, \beta_2+2\delta).
\]
The proof follows, as in Lemma~\ref{l1}, comparing the degrees of the four terms in both sides of equation \eqref{e1R}. Notice that the only difference is that the first degree of the above list is one larger than the corresponding list given in \eqref{eq:4degrees}.
\end{proof}

Next two results are the equivalent versions of Lemmas~\ref{l2} and
\ref{l4} for trigonometric Riccati equations. We only prove the
second one because the proof of the first one is essentially the
same.

\begin{lemma}\label{l2-1}
If a trigonometric polynomial equation~\eqref{e1R} of degree $\eta$
has a trigonometric polynomial solution, then it is equivalent to an
equation with $B_0(\theta)\equiv 0$ and the same degree:
\begin{equation}\label{e1*-1}
A(\theta) Y'=B_1(\theta)Y+B_2(\theta)Y^2.
\end{equation}
Moreover, each trigonometric polynomial solution of equation~\eqref{e1R} corresponds to a
trigo\-nometric polynomial solution of the new equation~\eqref{e1*-1}.
\end{lemma}

\begin{lemma}\label{l4-1}
If $Y_1(\theta)$ is a nonconstant real trigonometric polynomial
solution of equation \eqref{e1*-1}, set
\[
Y_1(\theta)=\frac{y_1(x)}{(1+x^2)^{\eta_1}}\quad \text{ with } \quad \gcd(y_1(x), 1+x^2)=1,
\]
where $\eta_1$ is the degree of $Y_1$ and $\deg y_1\le 2\eta_1$,
then any irreducible factor of $y_1(x)$ is a factor of the
polynomial $a(x)$ defined in \eqref{eq:defab}.
\end{lemma}

\begin{proof} From the transformation $\Phi$, we get
\[
x'=\frac{dx}{d\theta}=\frac{1+x^2}{2},
\]
it follows that
\[
{Y_1}'(\theta)=\frac{\dot y_1(x)(1+x^2)-2 \eta_1 x y_1(x)}{2(1+x^2)^{\eta_1}},
\]
where the dot and prime denote the derivative with respect to $x$
and $\theta,$ respectively. So equation \eqref{e1*-1} becomes
\[
\frac{a(x)}{2(1+x^2)^{\alpha}}\left(\dot y_1(x)(1+x^2)-2\eta_1 x y_1(x)\right) =
\frac{b_1(x) y_1(x)}{(1+x^2)^{\beta_1}} +
\frac{b_2(x)(y_1(x))^2}{(1+x^2)^{\beta_2+\eta_1}}.
\]
This last equality shows that each irreducible factor of $y_1(x)$ is
a factor of $a(x)$, where we have used the fact that
$\gcd(y_1(x), 1+x^2)=1$. The lemma follows.
\end{proof}

Next five lemmas split the essential difficulties for proving Theorem~\ref{t2}.

\begin{lemma}\label{l3-1} $($i\,$)$ Assume that equation \eqref{e1*-1} has at least three different
trigonometric polynomial solutions $Y_0(\theta)\equiv 0, \,
Y_1(\theta), \, Y_2(\theta)$. Set
\begin{equation}\label{e226}
 Y_1(\theta)=\frac{y_1(x)}{(1+x^2)^{\eta_1}}, \quad
Y_2(\theta)=\frac{y_2(x)}{(1+x^2)^{\eta_2}},
\end{equation}
where $\eta_1=\deg Y_1, \, \eta_2=\deg Y_2$, $\deg y_i\le 2\eta_i,$
$\eta_1\le \eta_2$ and $\gcd(y_i(x), 1+x^2)=1$ for $i=1, 2$. Except
the solution $Y_0(\theta)\equiv 0$, all the other solutions of
equation \eqref{e1*-1} can be expressed as
\begin{align*}
Y(\theta;c)=&\frac{y_1(x) \, y_2(x) }{c y_1(x) (1+x^2)^{\eta_2}+(1-c) y_2(x)(1+x^2)^{\eta_1}}\\
=&\frac{g(x)\, \tilde y_1(x) \, \tilde y_2(x) }{(1+x^2)^{\eta_1}\big(c \tilde y_1(x)
 (1+x^2)^{\eta_2-\eta_1}+(1-c) \tilde y_2(x)\big)}, \quad c\in\R,
\end{align*}
where $g=\gcd( y_1, y_2),$ i.e. $y_i= g\, \tilde y_i$, with $\gcd(
\tilde y_1, \tilde y_2)=1$.

\medskip

\noindent $($ii$\,)$ Assume that equation \eqref{e1*-1} has at least
four different trigonometric polynomial solutions. Then we can
always choose two of them, say $Y_1(\theta)$ and $Y_2(\theta)$, of
the same degree $\eta_1=\eta_2$, and then all the solutions of
\eqref{e1*-1}, except $Y_0(\theta)\equiv 0$, are
\begin{equation}\label{e123*-11}
Y(\theta;c)=\frac{y_1(x) \, y_2(x) }{(1+x^2)^{\eta_1}\big(c y_1(x)
+(1-c) y_2(x)\big)}= \frac{g(x)\, \tilde y_1(x) \, \tilde y_2(x)
}{(1+x^2)^{\eta_1}r_c(x)},
 \quad c\in\R, \end{equation}
where $r_c(x):=c \tilde y_1(x)+(1-c) \tilde y_2(x).$
\end{lemma}

\begin{proof}
$(i)$ It follows similarly as Lemma~\ref{l3} and the expressions
\eqref{e226} of $Y_1(\theta)$ and $Y_2(\theta)$. The details are
omitted.

\medskip

\noindent $(ii)$ We can assume that we are in the situation of item $(i)$ with $\eta_1<\eta_2$,
because if $\eta_1=\eta_2$, we are done. Observe that
\[
Y(\theta;0)=Y_1(\theta)\quad\text{and}\quad Y(\theta;1)=Y_2(\theta),
\]
where none of the polynomials $g, \, \tilde y_1$ and $\tilde y_2$
can have the factor $1+x^2.$

Therefore in order for $Y(\theta;c)$ to be a real trigonometric polynomial solution of equation
\eqref{e1*-1} we must have $c=0$, or $c=1$, or $c\ne 0, \, 1 $ and it is such that
\begin{itemize}
\item[$(a)$] either $ s_c (x):=c \tilde y_1(x) (1+x^2)^{\eta_2-\eta_1}+(1-c) \tilde y_2(x)$
has no the factor $1+x^2$ and it divides $g(x)$;
\item[$(b)$]or $s_c (x)=\widehat{s}_c(x)(1+x^2)^\sigma$
with $0<\sigma\in\N,$ $\gcd(\widehat{s}_c(x), 1+x^2)=1$ and $\widehat{s}_c(x)$ divides $g(x)$.
\end{itemize}
Let us prove that case $(b)$ never happens. Otherwise, notice that since $\eta_2-\eta_1>0,$ the
polynomials $s_c(x)$ and $c \tilde y_1(x) (1+x^2)^{\eta_2-\eta_1}$, both have the factor
$(1+x^2)$. Then $\tilde y_2$ would also have the factor $1+x^2$, in contradiction with its
definition.

Therefore, only case $(a)$ happens and the degree of $Y(\theta, c), c\ne 1$ must be $\eta_1$. As a
consequence, taking one of the solutions $Y(\theta;c), c\ne0, 1$, together with $Y_1(\theta)$ both
have the same degree, that is $\eta_2=\eta_1$. Thus the expression proved in item $(i)$ reduces
to \eqref{e123*-11}, as want to prove.
\end{proof}

\begin{lemma}\label{l-new} Assume that equation \eqref{e1*-1} has at
least four different trigonometric polynomial solutions and all the
notations of Lemma~\ref{l3-1}. Then all the trigonometric polynomial
solutions of this equation, different from $Y_0(\theta)\equiv 0$,
have degree $\eta_1,$ except maybe one that can have higher degree.
If this solution exists, then it corresponds to a unique value of
$c,$ $c=\breve c$ such that
\[
r_{\breve c}(x) ={\widehat{r}_{\breve c}}(x) (1+x^2)^\nu, \quad 0<\nu\in\N.
\]
Moreover, if $Y(\theta;c),$ $c=0, 1, c_1, c_2\ldots, c_k$ denote all
its trigonometric polynomial solutions of degree $\eta_1$ and
$Y(\theta;\breve c)$ its trigonometric polynomial solution of higher
degree $\eta_1+\nu$ $($if it exists$)$. Then
\begin{equation}\label{factor-g}
g(x)=\widehat{r}_{\breve c}(x) \prod_{j=1}^k r_{c_j}(x) \check g(x),
\end{equation}
for some polynomial $\check g.$ Furthermore, for all $r_{c_j}$ except
maybe one of them, say $r_{c_k}=r_{\hat c}$ which can be a nonzero
constant polynomial, it holds that
\begin{equation}\label{e-rho}
\deg r_{c_j}=\rho\ge 1, \quad j=1, 2, \ldots, k-1.
\end{equation}
\end{lemma}

\begin{proof}
To study the degrees of all the trigonometric polynomial solutions,
observe that from equation \eqref{e123*-11} given in Lemma
\ref{l3-1} all the solutions have degree $\eta_1$ except the ones
corresponding to the values of $c$ such that $r_c$ has the factor
$1+x^2$ and for these values of $c$ their degrees will be greater
than $\eta_1$. Let us prove that this value of $c$, if exists, is
unique. If there were two values of $c\not\in\{0, 1\}$, say $c_1\ne
c_2$ then $r_{c_1}$ and $r_{c_2}$ evaluated at $x=\pm \mathbf i$
would vanish simultaneously. Some simple computations will give that
also both $\tilde y_1$ and $\tilde y_2$ would vanish at $x=\pm
\mathbf i$, a contradiction with $\gcd(\tilde y_1, \tilde y_2)=1.$

In fact, the above reasoning proves that for $c_1\ne c_2$, the
corresponding $r_{c_1}$ and $r_{c_2}$ have no common roots. This
shows that \eqref{factor-g} holds. Hence the lemma follows.
\end{proof}

\begin{lemma}\label{l5-1}
Assume that equation \eqref{e1*-1} has at least four different
trigonometric polynomial solutions and all the notations of
Lemmas~\ref{l3-1} and~\ref{l-new}. Denote by $N(g)$ the number of
different factors of $g(x)$ decomposed into linear complex
polynomial factors and let $\rho$ be given in \eqref{e-rho}. The
following statements hold.
\begin{itemize}
\item[$(i)$] If for any $c\in\R\setminus\{0, 1\}$, $r_c$ is never a constant multiple of $(1+x^2)^\sigma$, where $0\le\sigma\in\N,$ then equation \eqref{e1*-1} has at most $\min\big(N(g)+3, [\deg(g)/\rho]+3\big)$ trigonometric polynomial solutions, where $[\, \, ]$ denotes the integer part function.
\item[$(ii)$] If there is exactly one $c\not\in\{0, 1\}$ such that $r_c$ is a nonzero constant multiple of $(1+x^2)^\sigma$, with $0\le\sigma\in\N,$ then equation \eqref{e1*-1} has at most $\min\big(N(g)+4, [\deg(g)/\rho]+4\big)$ trigonometric polynomial solutions.
\item[$(iii)$] If there is one $c,$ $c=\breve c$ such that $r_{\breve c}(x)=p (1+x^2)^\sigma, 0<\sigma\in\N$ and one $c$, $c=\hat c$ such that $r_{\hat c}(x)=q$ for some nonzero constants $p$ and $q$, with $\breve c, \hat c\not\in\{0, 1\}$, then equation \eqref{e1*-1} has at most $\eta+2$ trigonometric polynomial solutions.
\end{itemize}
\end{lemma}

\begin{proof}
By item $(ii)$ of Lemma~\ref{l3-1}, along this proof we can assume that all nonzero solutions of the Riccati equation \eqref{e1*-1} are given by the expression \eqref{e123*-11}.

\medskip

\noindent $(i)$ Set $r_c(x) = \widehat{r}_c(x)(1+x^2)^\nu$ with
$\nu$ a nonnegative integer and $\widehat{r}_c(x)$ a nonconstant
polynomial satisfying $\gcd(\widehat{r}_c(x),1+x^2)=1$. Since $g(x)$
is a real polynomial and $\gcd(g(x), 1+x^2)=1$, in order that
$g(x)/r_c(x)$ with $\nu=0$ or $g(x)/\widehat{r}_c(x)$ with $\nu\ne
0$, is a polynomial, each zero $x_0$ of $g(x)$ must be a zero of
$r_c(x)$, i.e.
\[
r_c(x_0)=c\tilde y_1(x_0)(1+x_0^2)^{\eta_2-\eta_1}+(1-c)\tilde y_2(x_0)=0.
\]
Since $\gcd(\tilde y_1, \tilde y_2)=1$ and $\gcd(\tilde y_2, 1+x^2)=1$, this last equation has a
unique solution $c_0$. Hence we have proved that there are at most $N(g)$ values of $c$ for which
$g(x)/r_c(x)$ with $\nu=0$ or $g(x)/\widehat{r}_c(x)$ with $\nu\ne 0$ can be real polynomials.
Hence the general solution \eqref{e123*-11} of equation \eqref{e1*-1} contains at most $N(g)+3$
real trigonometric polynomial solutions including the trivial one $Y_0=0$, and $Y_1, \, Y_2$. This
proves the first part of item $(i).$ The second bound given by $[\deg(g)/\rho]+3$ follows from
\eqref{factor-g} and~\eqref{e-rho}.

\medskip

\noindent $(ii)$ The proof of this item follows adding to the
maximum number of trigonometric polynomial solutions given in item
$(i)$ the extra one corresponding to this special value of $c.$

\medskip

\noindent $(iii)$ We have that $r_{\breve c}(x)=p \, (1+x^2)^\sigma,
0<\sigma\in\N$ and $r_{\hat c}(x)=q$, with $p, q$ nonzero real
numbers. Then
\begin{align*}
\tilde y_1(x)= \frac{(1-\hat c)\, p \, (1+x^2)^\sigma -(1-\breve c)\, q
}{\breve c-\hat c}, \quad \tilde y_2(x)= \frac{-\hat c \, p\,
(1+x^2)^\sigma +\breve c \, q }{\breve c-\hat c}.
\end{align*}
As a consequence
\[
r_c(x)=\frac{(c-\hat c)\, p\, (1+x^2)^\sigma -(c-\breve c)\, q }{\breve c-\hat c},
\]
and then $\deg r_c =2\sigma\ge2,$ for all $c\ne\hat c.$ Since we are
assuming the existence of the value $c=\breve c$ for which the
equation has a trigonometric polynomial solution of degree strictly
greater than the degree $\eta_1$ of all the other trigonometric
polynomial solutions, we know that $\eta_1\le\eta-\sigma\le \eta-1.$
By Lemmas~\ref{l3-1} and \ref{l-new},
\[
Y(\theta;c)= \frac{g(x)\, \tilde y_1(x) \, \tilde y_2(x)
}{(1+x^2)^{\eta_1}r_c(x)}=\frac{\check g(x)\, \tilde y_1(x) \, \tilde
y_2(x)\, \prod_{j=1}^k r_{c_j}(x) }{(1+x^2)^{\eta_1}r_c(x)},
\]
where $k$ is the number of trigonometric polynomial solutions different from $0, Y_1(\theta),$
$Y_2(\theta),$ $Y(\theta;\breve c)$ and $Y(\theta;\hat c)$. We claim that $k\le \eta-3.$ If the
claim holds then the maximum number of trigonometric polynomial solutions is $(\eta-3)+5=\eta+2$
as we wanted to prove.

To prove the claim, notice that by imposing that for $c=\hat c$, the function $Y(\theta;\hat c)$
is a trigonometric polynomial. Since $r_{\hat c}(x)=q$, we get that
\[
D:=\deg\Big( \check g\, \tilde y_1 \, \tilde y_2\, {\prod_{j=1}^k} r_{c_j} \Big)\le 2\eta_1\le
2\eta-2.
\]
By using that for all $ j=1,..k$, $\deg r_{c_j} \ge 2$ and that $\deg \tilde y_1=\deg \tilde
y_2\ge2,$ we get that $D\ge 2(k+2)$. Thus $2(k+2)\le 2\eta-2$ and $k\le \eta-3,$ as we wanted to
prove.
\end{proof}

\begin{lemma}\label{lemma-b2}
$(i)$ Consider an equation of the form \eqref{e1*-1} and all the
notations introduced in Lemma~\ref{l3-1}. Assume that it has two
trigonometric polynomial solutions $Y_1(\theta)$ and $Y_2(\theta)$,
both of degree $\eta_1$. Then the function $b_2(x)$ appearing in the
expression of $B_2(\theta)$ in the $x$-variables is
\begin{equation}\label{exp-b2}
b_2(x)=(1+x^2)^{\beta_2+\eta_1+1-\alpha}\frac{\big(\dot{\tilde y}_1(x) \tilde y_2(x)-\tilde
y_1(x)\dot{\tilde y}_2(x) \big)a(x)}{2\, g(x)\, \tilde y_1(x)\tilde y_2(x)(\tilde y_1(x)-\tilde
y_2(x))}.
\end{equation}

\medskip

\noindent $(ii)$ If $\eta_1=\eta$ and the polynomial $g=\gcd(y_1,
y_2)$ has degree greater than or equal to $2\eta-1$ then
$b_2(x)\equiv 0.$ As a consequence, there are no trigonometric
polynomial Riccati equations of the form \eqref{e1*-1} and degree
$\eta$ with $B_2(\theta)\not\equiv0$ having two trigonometric
polynomial solutions of degree $\eta$ and such that their
corresponding $g$ satisfies $\deg g \ge 2\eta-1.$
\end{lemma}

\begin{proof}
$(i)$ Recall that $Y_1(\theta)=y_1(x)/(1+x^2)^{\eta_1}$ and
$Y_2(\theta)=y_2(x)/(1+x^2)^{\eta_1}$ are the solutions of equation
\eqref{e1*-1}. Then, we have for $i=1, 2$,
\[
\frac{a(x)}{2(1+x^2)^{\alpha}}\left(\dot y_i(x)(1+x^2)-2\eta_1
xy_i(x)\right)=\frac{
b_1(x)y_i(x)}{(1+x^2)^{\beta_1}}+\frac{b_2(x)y_i^2(x)}{(1+x^2)^{\beta_2+\eta_1}}.
\]
Solving these two equations gives
\[
b_2(x)=(1+x^2)^{\beta_2+\eta_1+1-\alpha}\frac{ \left(\dot y_1(x)
y_2(x)-y_1(x) \dot y_2(x)\right)a(x)}{
2\, y_1(x)y_2(x)\left(y_1(x)-y_2(x)\right)}.
\]
By using that $y_i(x)=g(x)\, \tilde y_i(x), i=1, 2,$ the desired
expression for $b_2(x)$ follows.

\medskip

\noindent $(ii)$ Since $\deg Y_j=\eta_j=\eta$ it follows from
Lemma~\ref{l1-1} that $\deg B_2=\beta_2=0.$ Moreover since
$y_j=g\tilde y_j,$ $\deg g\ge 2\eta-1$ and $\deg y_j\le 2\eta$ it
holds that $\deg \tilde y_j\le 1.$ Furthermore, by Lemma~\ref{l4-1}
we get that $\deg a \ge 2\eta-1$ and as a consequence $\deg A
=\alpha=\eta.$ Putting all together, by using item $(i)$ and that
$B_2(\theta)$ must be a constant, $B_2(\theta)\equiv p\in\R,$ we get
that
\begin{equation}\label{const}
2\, p\, g(x)\, \tilde y_1(x)\tilde y_2(x)(\tilde y_1(x)-\tilde
y_2(x))=(1+x^2)\big(\dot{\tilde y}_1(x) \tilde y_2(x)-\tilde
y_1(x)\dot{\tilde y}_2(x) \big)a(x).
\end{equation}
Recall that the polynomial $g$ has no the factor $1+x^2$.
Moreover $\tilde y_1,$ $\tilde y_2,$ and $\tilde y_1-\tilde y_2$
have at most degree 1, so they neither have this quadratic factor.
Finally, notice that $\dot{\tilde y}_1(x) \tilde y_2(x)-\tilde
y_1(x)\dot{\tilde y}_2(x)$ is a real constant, which is not zero
unless both $\tilde y_j$ are of degree zero, because recall that
$\gcd(\tilde y_1, \tilde y_2)=1.$

Hence the above equation \eqref{const} has only the solution $p=0$ when both
$\tilde y_j$ are of degree zero and otherwise it is not possible.
\end{proof}

\begin{lemma}\label{lemma-final}
Consider an equation of the form \eqref{e1*-1} with $\eta\ge2,$ $B_2(\theta)\not\equiv0,$ having
at least four different trigonometric polynomial solutions, and all the notations introduced in
Lemma~\ref{l3-1}. Assume that it has two trigonometric polynomial solutions $Y_1(\theta)$ and
$Y_2(\theta)$, both of degree $\eta$. Moreover suppose that there exists $c_0\not\in\{0, 1\}$ such
that $r_{c_0}(x)= c_0 \tilde y_1(x)+(1-c_0)\tilde y_2(x)= q (1+x^2)^\sigma$ for some $0\ne
q\in\R,$ $0\le \sigma\in\N$ and that $\deg g\in\{2\eta-2, 2\eta-3\}$. Then the number of
trigonometric polynomial solutions of \eqref{e1*-1} is at most $\eta+2.$
\end{lemma}

\begin{proof} Since $\deg Y_j=\eta$, by Lemma~\ref{l1-1}, $\deg B_2=0$
and $b_2(x)\equiv p\ne0$. In this case, using the same ideas as
in the proof of item $(i)$ of Lemma~\ref{lemma-b2}, we get the
expression
\[
2\, p\, g(x)\, \tilde y_1(x)\tilde y_2(x)(\tilde y_1(x)-\tilde
y_2(x))=(1+x^2)^{\eta-\alpha+1}\big(\dot{\tilde y}_1(x) \tilde
y_2(x)-\tilde y_1(x)\dot{\tilde y}_2(x) \big)a(x).
\]
that is essentially the same that in \eqref{const}. Since none of
the functions $\tilde y_1,$ $\tilde y_2$, $g$ has the factor
$1+x^2$, by imposing that the above equality holds we get that there
exists a nonzero polynomial $k(x),$ of degree at most one, such that
$\tilde y_1(x)-\tilde y_2(x)= k(x) (1+x^2).$ This equality together
with the assumption $r_{c_0}(x)= c_0 \tilde y_1(x)+(1-c_0)\tilde
y_2(x)= q (1+x^2)^\sigma$ implies that
\[
\tilde y_2(x)= q(1+x^2)^\sigma-c_0 k(x) (1+x^2).
\]
Since $\gcd(\tilde y_2, 1+x^2)=1$ the above equality implies that $\sigma=0$ and as a consequence
\[
\tilde y_2(x)= q-c_0 k(x) (1+x^2)\quad\text{and}\quad \tilde
y_1(x)=q+(1-c_0) k(x) (1+x^2).
\]
Hence, by Lemma~\ref{l3-1} the
general solution of the Riccati equation is
\[
Y(\theta;c)=\frac{g(x) \tilde y_1(x) \tilde y_2(x)}{(1+x^2)^\eta
r_c(x)}=\frac{g(x) \tilde y_1(x) \tilde y_2(x)}{(1+x^2)^\eta
\big(q+(c-c_0) k(x) (1+x^2)\big)}.
\]
Notice that the solution corresponding to $c=c_0$ is never a polynomial because the degree of the
numerator is at least $2\eta-3+4>2\eta.$ Hence, following the notations of Lemma~\ref{l5-1}, for
$c\ne c_0,$ $\deg r_c=\rho\ge2,$ and then the maximum number of polynomial solutions is $[
\deg(g)/2]+3\le \eta+2$, where the $3$ corresponds to the solutions $Y=0,$ $Y_1(\theta),$ and
$Y_2(\theta)$.
\end{proof}

\begin{lemma}\label{lemma-final-2}
Consider an equation of the form \eqref{e1*-1} with $B_2(\theta)\not\equiv0,$ having at least four
different trigonometric polynomial solutions, and all the notations introduced in
Lemma~\ref{l3-1}. Assume that it has two solutions $Y_1(\theta)$ and $Y_2(\theta)$ of degree
$\eta\ge1$ and such that $\deg g= 2\eta-2$. Then  $\eta\ge2$ and the number of trigonometric
polynomial solutions of \eqref{e1*-1} is at most $\eta+2.$
\end{lemma}

\begin{proof}
Set
\[
y_1(x)=g(x)\tilde y_1(x), \quad y_2(x)=g(x)\tilde y_2(x), \quad a(x)=g(x)\tilde a(x)
\]
where $\gcd (\tilde y_1, \tilde y_2)=1,$ $\deg \tilde y_1, \, \deg \tilde y_2, \, \deg \tilde a
\le 2,$ and none of these three polynomials have the factor $1+x^2.$ Set $\tilde a (x)= a_2 x^2+
a_1x+a_0.$

Notice that it is not possible that $\deg \tilde y_1=\deg \tilde
y_2=0$ because by \eqref{exp-b2} in Lemma~\ref{lemma-b2} this fact
implies that $B_2(\theta)\equiv0$, a contradiction.

Since $\deg (Y_1)=\deg(Y_2)=\eta$ we know by Lemma~\ref{l1-1} that $B_2(\theta)\equiv
p\in\R\setminus\{0\}$. Then, by Lemma~\ref{lemma-b2},
\[
2\, p\, g(x)\, \tilde y_1(x)\tilde y_2(x)(\tilde y_1(x)-\tilde
y_2(x))=(1+x^2)^{\eta-\alpha+1}\big(\dot{\tilde y}_1(x) \tilde
y_2(x)-\tilde y_1(x)\dot{\tilde y}_2(x) \big)a(x).
\]
Moreover, since $g$ and $\tilde y_i$ have no the factor $1+x^2$, we
get that $\alpha=\eta$ and
\begin{equation}\label{rela}
\tilde y_1(x)-\tilde y_2(x)= q\, (1+x^2)
\end{equation}
for some $q\in\R\setminus\{0\}$ and the above equality simplifies to
\[
2\, p\, q\, \tilde y_1(x)\tilde y_2(x)=\big(\dot{\tilde y}_1(x) \tilde
y_2(x)-\tilde y_1(x)\dot{\tilde y}_2(x) \big)\tilde a(x).
\]
Using \eqref{rela} we get
\[
2\, p\, \big(\tilde y_2(x)+q(1+x^2)\big)\tilde y_2(x)=\big(2\, x\,
\tilde y_2(x)-(1+x^2)\dot{\tilde y}_2(x) \big)\tilde a(x),
\]
or equivalently,
\begin{equation}\label{cond-final}
\big(2\, p\, \tilde y_2(x) -2\, x\, \tilde a(x)\big)\tilde
y_2(x)=-(1+x^2)\big(\dot{\tilde y}_2(x) \tilde a(x)+2\, p\, q \tilde
y_2(x) \big).
\end{equation}
Since $\gcd(\tilde y_2, 1+x^2)=1,$ the above equality can only
happen if
\[
2\, p\, \tilde y_2(x) -2\, x\, \tilde a(x)=s(x)\, (1+x^2), \quad
\text{where} \quad s(x)= s_1 x+ s_0, \, s_0, s_1\in\R.
\]
From it we get that
\[
\tilde y_2(x)=\frac{s(x)}{2p}(1+x^2)+\frac x p\, \tilde a(x), \quad \tilde
y_1(x)=\Big(\frac{s(x)}{2p}+q\Big)(1+x^2)+\frac x p\, \tilde a(x),
\]
and moreover $s_1=-2 a_2$ because $\deg \tilde y_i\le 2, i=1, 2.$
Then
\[
r_c(x)= \Big(\frac{s(x)}{2p}+c\, q\Big) (1+x^2) +\frac x p\, \tilde a(x), \qquad \deg(r_c)\le2,
\]
and by Lemma~\ref{l3-1},
\[
Y(\theta;c)=\frac{g(x) \tilde y_1(x) \tilde y_2(x)}{(1+x^2)^\eta
r_c(x)}.
\]
Notice that for all $c\ne c_0:=-(2a_1+s_0)/(2pq),$ $\deg r_c=2$ and $\deg r_{c_0}<2$. Moreover
$\gcd(r_c(x),1+x^2)=1,$ because otherwise, $1+x^2$ would be a factor of $\tilde a$. Since the
degree of $g$ is $2\eta-2$, by item $(ii)$ of Lemma~\ref{l5-1} the maximum number of trigonometric
polynomial solutions is $(2\eta-2)/2+4=\eta+3.$

To reduce this upper bound by 1 we have to continue our study. First notice that if the degree of
$r_{c_0}$ is 1 then $r_{c_0}$ should also divide $g(x)$ and then
\[
Y(\theta;c)=\frac{\check g(x) r_{c_0}(x) \tilde y_1(x) \tilde y_2(x)}{(1+x^2)^\eta r_c(x)},
\]
with $\deg (\check g)=2\eta-3.$ As a consequence, $\eta\ge 2$. By using the same arguments that in the proof of Lemma~\ref{l5-1}
we get that the maximum number of trigonometric polynomials solutions in this case is at most
$[\deg(\check g)/2]+4= [(2\eta-3)/2]+4=\eta+2,$ where the 4 counts the solution $0$ and the ones
corresponding to $c\in\{0,1,c_0\}.$ Then the result is proved when $\deg(r_{c_0})=1.$

Finally, we will prove that there is no $c_0\not\in\{0, 1\}$ such that $\deg r_{c_0}=0$. Imposing
that $r_{c_0}$ is constant we get that
\[
s_0=-2(c_0\, p\, q+a_1), \quad a_2=a_0,\quad\text{and}\quad
r_{c_0}(x)=-a_1/p.
\]
Then $\tilde a(x)=a_0\, x^2+a_1\, x+a_0$ and $ s(x)=
-2\, a_0\, x-2(c_0\, p\, q+a_1).$ Substituting the above expressions
into the function
\[
W(x):=\big(2\, p\, \tilde y_2(x) -2\, x\, \tilde a(x)\big)\tilde
y_2(x)+(1+x^2)\big(\dot{\tilde y}_2(x) \tilde a(x)+2\, p\, q \tilde
y_2(x) \big),
\]
obtained from \eqref{cond-final}, we get that
\[
W(x)= \frac 2 p
(1+x^2)\Big(c_0\, (c_0-1)\, {p}^{2}{q}^{2}{x}^{2}+a_0\, a_1\, x+
 \left((c_0-1)\, p\, q +a_1 \right) \left( c_0\, p\, q+a_1
 \right)\Big)
\]
and we know that it must vanish identically. The only solutions compatible with our hypotheses are
either $a_1=p\, q$ and $a_0=c_0=0$ or $a_1=-p\, q$, $a_0=0$ and $c_0=1$. In both situations the
value $c_0$ is either $0$ or $1$ and hence there are no solutions different from $Y_1$ or $Y_2$
such that $r_{c_0}$ is constant.

Therefore, we have shown that in all the situations the Riccati equation has at most $\eta+2$
trigonometric polynomial solutions. Hence $\eta\ge 2$ and the lemma follows.
\end{proof}

\smallskip

\begin{proof}[Proof of the upper bound given in Theorem~\ref{t2} when $\eta\ge2$]
Recall that we want to prove that trigonometric Riccati equations of degree $\eta\ge 2$ have at
most $2\eta$ trigonometric polynomial solutions. By using Lemma~\ref{l2-1} we can restrict our
attention to equation~\eqref{e1*-1}. Moreover, we can assume that this equation has at least four
trigonometric polynomial solutions ($\eta\ge2$), because if not we are done. Therefore we are
always under the assumptions of Lemma~\ref{l5-1} and, apart of the solution $Y_0(\theta)\equiv 0$,
we can suppose that the equation has two trigonometric polynomial solutions $Y_1(\theta)$ and
$Y_2(\theta)$, both of degree $\eta_1\le \eta.$ Then the corresponding $y_1$ and $y_2$ given in
expressions~\eqref{e226} have degree at most $2\eta_1.$ By the definition of $g$ in
Lemma~\ref{l3-1} we have $\deg g\le 2\eta_1$ and so $N(g)\le 2\eta_1\le 2\eta$.

In item $(iii)$ of Lemma~\ref{l5-1} we have also proved that when the situation described there
happens (that is the existence of one $c,$ $c=\breve c,$ such that $r_{\breve c}(x)=p (1+x^2)^\sigma,
0<\sigma\in\N$ and one $c$, $c=\hat c,$ such that $r_{\hat c}(x)=q$ for some nonzero constants $p$
and $q$) then equation~\eqref{e1*-1} has at most $\eta+2$ trigonometric polynomial solutions.
Since for $\eta\ge 2$ it holds that $\eta+2\le 2\eta$ and we do not need to consider this
situation anymore. Hence, by items $(i)$ and $(ii)$ of Lemma~\ref{l5-1}, when $N(g)\le 2\eta-4$ we
have proved the upper bound given in the statement.

Moreover, since $\deg g \ge N(g)$, by using Lemma~\ref{lemma-b2} we
also know that the upper bound holds when $N(g)\in\{2\eta,
2\eta-1\},$ because either they correspond to $B_2(\theta)\equiv0$
or with a Riccati equation with at most 4 trigonometric polynomial
solutions.

We will prove the result for the remaining situations by a case by case study,
according whether $N(g)=2\eta-2$ or $N(g)=2\eta-3.$

Observe also that we never have to consider the situations where
$\deg \tilde y_1=\deg \tilde y_2=0$ because by \eqref{exp-b2} in
Lemma~\ref{lemma-b2} this fact implies that $B_2(\theta)\equiv0$, a
contradiction.

\smallskip

\noindent $\bullet$ [\emph{Case $N(g)=2\eta-2$}]: We have that
$\eta_1=\eta.$ Let $g_1(x), \ldots, g_{2\eta-2}(x)$ be the $2\eta-2$
different linear factors of $g(x)$. Again by Lemma~\ref{lemma-b2} we
do not need to consider the cases $\deg g\ge2\eta-1$. Then
$g(x)=g_1(x)\cdots g_{2\eta-2}(x)$. We are precisely under the
situation of Lemma~\ref{lemma-final-2} and the upper bound is at
most $\eta+2\le 2\eta,$ as we wanted to see.

\smallskip

\noindent $\bullet$ [\emph{Case $N(g)=2\eta-3$}]: By item $(i)$ of Lemma~\ref{l5-1}, if for any
$c\in\R$, $c\tilde y_1(x)+(1-c)\tilde y_2(x)$ is not a constant multiple of $(1+x^2)^\sigma$ with
$\sigma$ is a nonnegative integer, then equation~\eqref{e1*-1} has at most $2\eta$ real
trigonometric polynomial solutions. Therefore, we can assume that there exists a
$c_0\in\R\setminus\{0, 1\}$ such that $c_0\tilde y_1(x)+(1-c_0)\tilde y_2(x)=q (1+x^2)^\sigma$
with $q \ne 0$ a nonzero constant and $\sigma$ a nonnegative integer and that this $c_0$ is
unique, see Lemma~\ref{l-new} and item $(iii)$ of Lemma~\ref{l5-1}.

Notice that $\deg g\ge N(g)=2\eta-3$ and recall that we only need
to consider the cases $\deg g\le2\eta-2$. Then $\eta_1\in\{\eta, \eta-1\}$.

\smallskip

\noindent $\blacktriangleright$ [\emph{Subcase $\deg(g)=2\eta-2$ and $\eta_1=\eta-1$}]:
It never holds because these hypotheses imply that $\deg \tilde y_1=\deg \tilde y_2=0.$

\smallskip

\noindent $\blacktriangleright$ [\emph{Subcase $\deg(g)=2\eta-2$ and
$\eta_1=\eta$}]: We can assume that $\max(\deg \tilde y_1, \deg
\tilde y_2)>0$. Then, by Lemma~\ref{lemma-final} the maximum number
of trigonometric polynomial solution is $\eta+2\le 2\eta,$ for
$\eta\ge2,$ as we wanted to prove.

\smallskip

\noindent $\blacktriangleright$ [\emph{Subcase $\deg(g)=2\eta-3$ and
$\eta_1=\eta-1$}]: We know that $\max(\deg \tilde y_1, \deg \tilde
y_2)=1$. In fact, the existence of $c_0\not\in\{0, 1\}$ such that
$r_{c_0}(x) =q (1+x^2)^\sigma, q\ne0$ implies that $\sigma=0$ and
$\deg \tilde y_1=\deg \tilde y_2=1$. Now, by item $(ii)$ of
Lemma~\ref{l3-1},
\[
Y(\theta;c)=\frac{g(x) \, \tilde y_1(x) \, \tilde y_2(x)}{(1+x^2)^{\eta-1} r_c(x)}.
\]
In this situation, for this special value $c=c_0,$ the corresponding
$Y(\theta;c_0)$ is not a trigonometric polynomial because $\deg(g \,
\tilde y_1 \, \tilde y_2)=2\eta-3+2=2\eta-1>2(\eta-1).$ Then by item
$(i)$ of Lemma~\ref{l5-1} the number of trigonometric polynomial
solutions in this case is at most $N(g)+3=2\eta.$

\smallskip

\noindent $\blacktriangleright$ [\emph{Subcase $\deg(g)=2\eta-3$ and
$\eta_1=\eta$}]: Once more, we know that $\max(\deg \tilde y_1, \deg
\tilde y_2)=1$. This situation is again covered by
Lemma~\ref{lemma-final} and the maximum number of trigonometric
polynomial solutions is $\eta+2\le 2\eta,$ for $\eta\ge2.$
\end{proof}

\begin{proof}[Examples for proving Theorem~\ref{t2} with $\eta\ge 2$]
The above proofs show that equation~\eqref{e1*-1} can have $2\eta$
trigonometric polynomial solutions only in the case $N(g)=\deg
g=2\eta-3$.

\smallskip

Next we provide examples showing that there exist equations~\eqref{e1*-1} which have exactly $k+3$ trigonometric polynomial solutions with $k\in\{1, \ldots, 2\eta-3\}$. Taking $g(x)=\prod\limits_{i=1}\limits^{k}(x-c_i)$ with $c_i\in\R$ and $c_i\ne c_j$ for $1\le i\ne j\le k$. Set
\begin{equation}\label{epres}
A(\theta)=\frac{2(x-d_1)(x-d_2)g(x)}{(1+x^2)^\eta}, \quad
B_1(\theta)=\frac{g(x)h_1(x)+\dot{g}(x)h_2(x)}{(1+x^2)^\eta}, \quad
B_2(\theta)=-1,
\end{equation}
where $d_1$ and $d_2$ are two different real constants and are
different from $c_i$ for $i=1, \ldots, k$, and
$h_1(x)=(1+x^2)(2x-d_1-d_2)-2(\eta-1)x(x-d_1)(x-d_2)$ and
$h_2(x)=(1+x^2)(x-d_1)(x-d_2)$. We can check that equation~\eqref{e1*-1} with the prescribed $A, \, B_1, \, B_2$ given in \eqref{epres} has the trigonometric polynomial solutions
\[
Y_1(\theta)=\frac{g(x)(x-d_1)}{(1+x^2)^{\eta-1}}, \qquad
Y_2(\theta)=\frac{g(x)(x-d_2)}{(1+x^2)^{\eta-1}}.
\]
Then, by Lemma~\ref{l3-1},
\[
Y(\theta;c)=\frac{y_1(x)y_2(x)}{cy_1(x)+(1-c)y_2(x)}=\frac{\prod\limits_{i=1}\limits^{k}
(x-c_i)(x-d_1)(x-d_2)}{(1+x^2)^{\eta-1}(x-(cd_1+(1-c)d_2))}.
\]
Taking
\[
c=0, \quad 1, \quad \frac{c_i-d_2}{d_1-d_2}, \quad i=1, \ldots, k,
\]
we get $k+3$ trigonometric polynomial solutions of equation~\eqref{e1*-1} counting also the
trivial solution $Y_0=0$. Clearly for any other $c$, the solution $Y(\theta;c)$ cannot be a
trigonometric polynomial.

Of course, in the above construction if we take $g$ a nonzero
constant, then the equation has exactly three trigonometric
polynomial solutions.
\end{proof}

\begin{proof}[Proof of Theorem~\ref{t2} when $\eta =1$]
The simple Riccati equation
\[
\sin\theta\, Y'= 2\cos\theta \, Y- Y^2,
\]
has three trigonometric polynomial solutions $Y=0,$ $Y=1+\cos\theta,$ and $Y=-1+\cos\theta.$ Hence we know that when $\eta=1$ the number of
trigonometric polynomial solutions is at least~3.

Let us prove that~3 is also the upper bound. Otherwise, assume that there are Riccati equations
with $\eta=1$ and four trigonometric polynomial solutions to arrive to a contradiction. Therefore,
we are under the hypotheses of item $(ii)$ of Lemma~\ref{l3-1} and we can suppose that the
equation has two nonzero solutions of the same degree $\eta_1\in\{0, 1\}.$ If $\eta_1=0$, then
$y_1$ and $y_2$ are both constants. By Lemma~\ref{lemma-b2}, this forces that $B_2(\theta)\equiv
0$, a contradiction.

Hence we can assume that $\eta_1=1.$ By Lemma~\ref{l1-1},
$B_2(\theta)\equiv p\ne0$. By using once more the same ideas that in
the proof of item $(i)$ of Lemma~\ref{lemma-b2}, we get that
\[
2\, p\, g(x)\, \tilde y_1(x)\tilde y_2(x)(\tilde y_1(x)-\tilde
y_2(x))=(1+x^2)^{2-\alpha}\big(\dot{\tilde y}_1(x) \tilde
y_2(x)-\tilde y_1(x)\dot{\tilde y}_2(x) \big)a(x).
\]
Then $\alpha=1$, and $\tilde y_1-\tilde y_2$ is a constant multiple of $1+x^2$. In particular, one
of $\tilde y_1$ and $\tilde y_2$ has degree $2$ and so $\deg g=0$. Therefore, we are under the
hypotheses of Lemma~\ref{lemma-final-2} and we get that $\eta\ge2,$ again a contradiction. Then
the result follows.
\end{proof}

\section*{Acknowledgments}
The first and second authors are partially supported by the MINECO
MTM2013-40998-P and AGAUR 2014SGR568 grants. The third author is
partially supported by the NNSF of China 11271252 and the Innovation
program of Shanghai Municipal Education Commission of China 15ZZ02.
All authors are partially supported by the European Community
FP7-PEOPLE-2012-IRSES-316338 grant.

\end{document}